\documentclass[12pt,a4paper,reqno,oneside]{amsart}
\usepackage{t1enc}
\usepackage{times}
\usepackage[mathscr]{euscript}
\usepackage{graphicx}

\usepackage{hyperref}

\usepackage{amsmath, amsthm, amsfonts, amssymb}

\usepackage{xcolor}

\theoremstyle{plain}
\newtheorem{thm}{Theorem}[section]
\newtheorem*{thm*}{Theorem}
\newtheorem{lem}{Lemma}[section]

\newtheorem{corl}{Corollary}[section]
\theoremstyle{definition}

\theoremstyle{remark}
\newtheorem{remk}{Remark}[section]

\newcommand{\1}{1\!\!\,{\rm I}}

\newcommand{\wt}{\widetilde}

\newcommand{\be}{\begin{equation}}
\newcommand{\ee}{\end{equation}}
\newcommand{\bel}{\begin{equation}\label}
\newcommand{\ba}{\begin{aligned}}
\newcommand{\ea}{\end{aligned}}

\renewcommand{\lg}{\langle}
\newcommand{\rg}{\rangle}

\newcommand{\ve}{\varepsilon}
\newcommand{\vf}{\varphi}

\newcommand{\mbR}{{\mathbb R}}
\newcommand{\mbN}{{\mathbb N}}

\newcommand{\cF}{{\mathcal F}}

\newcommand{\supp}{\supp{\rm supp}}

\newcommand{\E}{\mathrm{E}}
\newcommand{\Pb}{\mathrm{P}}

\newcommand{\pto}{\stackrel{\mathrm{P}}{\to}}

\begin{document}
\title{On a skew stable L\'{e}vy process}
\author{Alexander Iksanov}
\address{A.Iksanov~~ Taras Shevchenko National University of Kyiv, Ukraine; e-mail: iksan@univ.kiev.ua}
\author{Andrey  Pilipenko}
\address{A. Pilipenko~~ Institute of Mathematics of Ukrainian National Academy of Sciences; Igor Sikorsky Kyiv Polytechnic Institute; e-mail: pilipenko.ay@gmail.com}

\begin{abstract}
The skew Brownian motion is a strong Markov process which behaves like a Brownian motion until hitting zero and exhibits an asymmetry at zero. We address the following question: what is a natural counterpart of the skew
Brownian motion in the situation that the noise is a stable L\'{e}vy process with finite mean and infinite variance. We define a skew stable L\'{e}vy process $X$ as the limit of a sequence of stable
L\'{e}vy processes which are perturbed at zero. We point out a formula for the resolvent of $X$ and show that $X$ is a solution to a stochastic differential equation with a local time. Also, we provide a representation of $X$ in terms of It\^{o}`s excursion theory.
\end{abstract}
\keywords{Excursion theory; functional limit theorem; resolvent; recurrent extension of a Markov process; skew Brownian motion; stable L\'{e}vy process; stochastic differential equation with a local time}
\subjclass[2020]{Primary: 60F17, 60J35; Secondary: 60J50}

\maketitle

\section{Introduction and main results}\label{intro}

A skew Brownian motion (SBM) with parameter $p\in [0,1]$ appears in the book \cite{ItoMcKean65} as a diffusion that behaves like a Brownian motion until hitting $0$ and whose excursions select the positive or negative sign with probabilities $p$ and $1-p$, respectively. There are numerous alternative descriptions and generalizations of the SBM. Portenko constructed an SBM as a solution to a stochastic differential equation (SDE) with a generalized drift or a diffusion with a semipermeable membrane \cite{Portenko90, portenko1974Irregular}. Harrison and Shepp \cite{Harrison+Shepp} proved that an SBM is a strong
solution to an SDE with a local time drift, which can be thought of as Dirac's delta function drift in Portenko's framework. Harrison and Shepp 
also constructed an SBM as an appropriate limit of random walks perturbed at $0$, see also \cite{IksanovPilipenko2016Perturbed, MinlosZhizhina, NgoPeigne, PilipenkoPrykhodkoUMZh} and references therein. Walsh \cite{Walsh:1978} described an SBM with the help of a  martingale problem, see also \cite{BarlowPitmanYor}. The detailed review on the SBM is presented in \cite{Lejay:skew}.

In this paper we address the following question: what is a natural analogue of the SBM in the case of a stable noise?
For $\alpha\in (1,2)$, denote by $U_\alpha:=(U_\alpha(t))_{t\geq 0}$ a symmetric $\alpha$-stable L\'{e}vy process with characteristic function $$\E\exp({\rm i}zU_\alpha(t))= \exp(-t|z|^\alpha),\quad z\in\mbR,\ t\geq 0.$$
This process hits any point  with probability 1, see, for instance, \cite[p. 63]{Bertoin:1996} or \cite[Example 43.42]{Sato}. We are going to construct a Feller process that behaves like $U_\alpha$ until hitting $0$ and have some ``asymmetry'' at the origin. To understand what could be a natural asymmetry at $0$ we attempt to find analogies with the SBM construction. Note that $U_\alpha$ does not hit $0$ by a single jump and cross the zero level infinitely often before hitting $0$, see \cite[Theorem 6.4]{Watanabe1962}. It does not also exit $0$ by a jump and changes sign infinitely often on exiting $0$ \cite[Theorem 47.1]{Sato}. Thus, unlike in the case of SBM, selecting signs of the excursions becomes an issue.

The transition probability density function of the SBM with parameter $p$ at time $t>0$ is given by
$$(x,y)\mapsto \vf_t^{(2)}(x-y)+(2p-1) {\rm sgn}\, (y)\vf_t^{(2)}(|x|+|y|),\quad x,y\in\mathbb{R},$$ where $\vf_t^{(2)}$ is the density of a centered normal distribution with variance $t$. For $t>0$, denote by $\vf_t^{(\alpha)}$ the density of the random variable $U_\alpha(t)$, $\alpha\in (1,2)$. It turns out that, for $p\in [0,1]$, $$(x,y)\mapsto \vf^{(\alpha)}_t(x-y)+(2p-1){\rm sgn}\, (y)\vf^{(\alpha)}_t(|x|+|y|),\quad x,y\in\mathbb{R}$$ is the transition density of a Markov process, whose martingale characterization is given in \cite{PortenkoSkewLevy}.
Now we provide an informal construction of this process and note that it does not behave like $U_\alpha$ outside $0$.
The jumps of the process are governed by the intensity function proportional to $ |x|^{-(1+\alpha)}$, $x\in\mathbb{R}$. These are accumulated until the process changes sign.  At the epoch of sign change a new sign is selected positive
with probability $p$ and negative with probability $1-p$.

Existence and uniqueness of a strong solution to an SDE with a local time were proved by Harrison and Shepp with the help of Tanaka's formula and Nakao's theorem. This technique fails in the case of an $\alpha$-stable noise, and so do
arguments related to space- and time-change transforms, which are efficient for one-dimensional diffusions with local times, see \cite{LeGall83}. When constructing a diffusion with a generalized drift
Portenko uses a partial differential equations approach. An essential part of his proof is based on the result dealing with a
jump of the normal derivative of a single layer potential. It turns out that a similar result,
in which the derivative has to be replaced with some nonlocal operator, holds true for a potential generated by the process $U_\alpha$. Unfortunately, such an approach leads to strongly continuous semigroups without nonnegativity condition \cite{PortenkoLoebus, PortenkoOsypchuk1994}.

We also mention here several recent results on strong solutions of SDEs with
singular drifts and additive fractional Brownian motion noise having a small
Hurst parameter, see \cite{BanosOrtizProskePilipenko, CatellierGubinelli} and references therein.
These results are derived under the assumption that the noise has a local time,
 which is sufficiently smooth with respect to the spatial parameter.
Observe that the $\alpha$-stable L\'{e}vy process $U_\alpha$ does not enjoy such a property.

To define the asymmetry at $0$ in a natural way, we shall use the approach
of Harrison and Shepp. Specifically, 
we construct perturbations of the $\alpha$-stable process $U_\alpha$ as follows. Let $\zeta_1$, $\zeta_2,\ldots$ be independent copies of a random variable $\zeta$, which are independent of $U_\alpha$. Assume that $\zeta\neq 0$ almost surely (a.s.). With these at hand, we define the process $X_\zeta:=(X_\zeta(t))_{t\geq 0}$ which satisfies $X_\zeta(0)=x\neq 0$ and has the same increments as $U_\alpha$ on the time intervals where $X_\zeta$ does not ''touch'' $0$. Upon the $k$th touch of $0$ the process $X_\zeta$ has jump $\zeta_k$. To make the previous discussion formal, put
\bel{eq:def_of_X}
\ba
\sigma_0=0,\quad \sigma_{k+1}:=\inf\{t>\sigma_k\ : \ \zeta_k+ U_\alpha(t)-U_\alpha(\sigma_k)=0\},\\
X_\zeta(t):=x+ U_\alpha(t),\quad t\in[0, \sigma_{1}),\\
X_\zeta(t):=\zeta_k+ U_\alpha(t)-U_\alpha(\sigma_k),\quad t\in[\sigma_k, \sigma_{k+1}),\quad k\geq 1.
\ea
\ee
Now we want to successively decrease the perturbations $(\zeta_k)_{k\geq 1}$ of $U_\alpha$. To this end, for each positive integer $n$, replace $(\zeta_k)_{k\geq 1}$ with $(\zeta_k/n)_{k\geq 1}$, then define the processes $X_{\zeta/n}$ and send $n\to\infty$. Note that each particular perturbation tends to $0$. Nevertheless, the smaller the jump from $0$ is, the smaller the return time to $0$ is. As a consequence, the number of visits to $0$ increases as $n$ grows.

In this paper we aim at finding a distributional limit (that we denote by $X$) for the sequence $(X_{\zeta/n})_{n\geq 1}$ and investigating its properties. In particular, we shall point out the resolvent, the entrance law, the excursions measure and an SDE that $X$ satisfies.
 Related to our investigation are the papers \cite{LambertSimatos, Yano2008convergence, Yano2015functional}, in which some invariance principles are obtained in terms of convergence of the excursions measures.

Before formulating our results we note that, for $t>0$ and large $n$, the value $X_{\zeta/n}(t)$ is the sum of $U_\alpha(t)$ and a random number (depending on $t$) of small perturbations from the collection $(\zeta_k/n)_{k\geq 1}$. Intuitively, if the distribution of $\zeta$ is light-tailed, the contribution of the perturbations should be negligible. We shall show below that this (trivial) situation occurs whenever $\E|\zeta|<\infty$. Assume now that the opposite situation prevails, which particularly means that $\E|\zeta|=\infty$. Then, to make the sum of perturbations, properly normalized, convergent, it is natural to assume that the distribution of $\zeta$ belongs to the domain of attraction of a $\beta$-stable distribution with $\beta\in (0,1)$ (here, $\beta$ could have been equal to $1$; however, we do not treat this case in the present paper). This means that the variables $\zeta_1+\ldots+\zeta_n$, properly normalized and centered, converge in distribution to a random variable with a $\beta$-stable distribution. It is known that this happens if, and only if, the function $x\mapsto \Pb(|\zeta|>x)$ is regularly varying at $+\infty$ of index $-\beta$ and the limits
\bel{eq:c+-}
c_\pm:=\lim_{x\to +\infty}\frac{\Pb(\pm \zeta >x)}{\Pb(|\zeta|>x)}
\ee
exist and satisfy $c_-+c_+>0$.

Given next are some notation to be in force throughout the paper. For a process $Y$, denote by $\sigma$ or $\sigma(Y)$ the first hitting time of $0$, that is,
$$\sigma(Y)=\sigma:=\inf\{ t>0\ :\ Y(t)=0\}.$$ For bounded measurable functions $f:\mathbb{R} \to \mathbb{R}$, put
\[
V_\lambda f(x):= \E^x \int_0^\sigma e^{-\lambda s}f(U_\alpha(s)){\rm d}s,
\]
so that $V_\lambda$ is the resolvent of $U_\alpha$ killed at $0$.
We write $D:=D[0,\infty)$ for the Skorokhod space of c\`{a}dl\`{a}g functions defined on $[0,\infty)$. We always assume that the space $D$ is endowed with the $J_1$-topology. 
For a measurable function $f:\mathbb{R}\to\mathbb{R}$ and a measure $\nu$ on $\mathbb{R}$ we write $\lg \nu, f\rg$ for $\int_{\mathbb{R}}f(x)\nu({\rm d}x)$ provided that the integral is well-defined. In particular, if $\nu$ is a probability measure, then $\lg \nu, f\rg=\E f(\tau)$, where $\tau$ is a random variable with distribution $\nu$.

We are ready to formulate our main results.

\noindent {\bf Theorem A.}
\textit{
Assume that either the function $x\mapsto \Pb(|\zeta|>x)$ is regularly varying at $+\infty$ of index $-\beta$, $\beta\in (0,1)$ and \eqref{eq:c+-} holds, or $\E|\zeta|<\infty$. If $X_{\zeta/n}(0)$ converges in distribution as $n\to\infty$ to some random variable $\xi$, then the processes $X_{\zeta/n}$ converge in distribution on $D$ to a Feller process $X$ starting at $\xi$.}

\noindent \textit{(a) If  $\beta<\alpha-1$,
 then the resolvent of $X$ is given by
\bel{eq:RES}
R_\lambda f(x) =V_\lambda f(x)+\E^x e^{-\lambda \sigma(U_\alpha)} \frac{\lg  \eta^\ast, V_\lambda f\rg}{\lg  \eta^\ast, V_\lambda 1\rg} 
\ee
for  bounded measurable  $f$,
where $\eta^\ast$ is a measure defined by
\begin{equation}\label{eq:measure_eta}
\eta^\ast({\rm d}x)=(c_-\1_{(-\infty, 0)}(x)+c_+\1_{(0,\infty)}(x))|x|^{-(1+\beta)}{\rm d}x,\quad x\in\mathbb{R},
\end{equation}
and the constants $c_\pm$ are given in \eqref{eq:c+-}.}

\noindent \textit{
(b) If $\beta>\alpha-1$ or $\E |\zeta|<\infty$,
then $ X(t)= \xi+U_\alpha(t)$, $t\geq 0$, where $\xi$ and $U_\alpha$ are independent.
}

One of the standing assumptions of the previous theorem is $\alpha\in (1,2)$. Put formally $\alpha=2$, so that the noise becomes a Brownian motion $W$, say. Then, under the assumption $\E |\zeta|<\infty$, a counterpart of the limit process in Theorem A, still denoted by $X$, is a SBM which solves the SDE
\bel{eq:SBM}
{\rm d}X(t) = {\rm d} W(t) + \gamma {\rm d}L(t).
\ee
Here, $\gamma=\frac{\E\zeta}{\E |\zeta|}\in [-1,1]$, and $L$ is a two-sided local time of $X$ at $0$. The claim can be justified with the help of arguments given in \cite{ChernyShiryaevYor} or \cite{IksanovPilipenko2016Perturbed}. Alternatively, this can be shown along the lines of the proof of Theorem A.

Although a resolvent uniquely determines the corresponding Feller process, formula \eqref{eq:RES}, being rather implicit, does not shed much light on the properties of $X$. As a remedy, we characterize a Feller process with resolvent \eqref{eq:RES} as a solution to an SDE.

\noindent {\bf Theorem B.}
\textit{Assume that the function $x\mapsto \Pb(|\zeta|>x)$ is regularly varying at $+\infty$ of index $-\beta$, $\beta\in (0,\alpha-1)$ and \eqref{eq:c+-} holds. Let $X$ be a Feller process with resolvent \eqref{eq:RES}.
Then $X$ is a weak solution to the SDE
\bel{eq:SKEW_L}
X(t)=X(0)+U_\alpha(t)+ S_\beta(L^X_0(t)),\quad t\geq 0.
\ee
Here, $L^X_0$ is a Blumenthal-Getoor local time of $X$ at $0$,
$S_\beta$ is a $\beta$-stable L\'{e}vy process which is independent of $U_\alpha$ and has the L\'{e}vy measure $\eta$ given by \begin{equation}\label{eq:Levy measure}
\eta({\rm d}x)=C(c_-\1_{(-\infty, 0)}(x)+c_+\1_{(0,\infty)}(x))|x|^{-(1+\beta)}{\rm d}x,\quad x\in\mathbb{R},
\end{equation}
where the constants $c_\pm$ are given in \eqref{eq:c+-}, $$C:=\Big(\int_0^\infty \E^x(1-e^{-\sigma(U_\alpha)})\eta({\rm d}x)\Big)^{-1}=\frac{\beta\sin\frac{\pi(\beta+1)}{\alpha}}{(c_-+c_+)\Gamma(1-\beta) \cos\frac{\pi\beta}{2}\sin\frac{\pi}{\alpha}}$$
and $\Gamma$ is the Euler gamma function.}

\textit{Furthermore, the process $X$ has a zero sojourn at $0$ with probability $1$.}
\begin{remk}
The definition of the Blumenthal-Getoor local time can be found in \cite[Theorem 2.3, Chapter V Section 2]{Blumenthal2012excursions} or Section \ref{sec2} below.
\end{remk}
\begin{remk}
The weak solution in Theorem A is a triple $(X, U_\alpha, S_\beta)$, with all components being defined on a common probability space, which satisfies equality \eqref{eq:SKEW_L} a.s. Here, the components are as defined in Theorem A. In particular,  $U_\alpha$ and $S_\beta$ are independent.

While not discussing a filtration, we only mention that it follows from the construction that the processes $U_\alpha$ and  $(S_\beta(L_0^X(t)))_{t\geq 0}$ are $(\cF^X_t)_{t\geq 0}$-adapted, where $(\cF^X_t)_{t\geq 0}$ is a filtration generated by $X$ and augmented by events of probability $0$. Observe that the process $S_\beta$ is not $(\cF^X_t)_{t\geq 0}$-adapted.
\end{remk}
Comparing equations \eqref{eq:SBM} and \eqref{eq:SKEW_L} we find it reasonable to call the process $X$ with resolvent \eqref{eq:RES} a {\it skew $\alpha$-stable L\'{e}vy process}.

Equation \eqref{eq:SBM} has a unique solution if $|\gamma|\leq 1$ and has no solution if $|\gamma|>1$, see \cite{Harrison+Shepp}. An interesting problem is to find a counterpart of the parameter $\gamma$ for equations like   \eqref{eq:SKEW_L}. Theorems C and D given next provide a solution to the problem as well as 
a description of the corresponding processes 
with the help of It\^{o}'s excursion theory. We shall recall basic definitions and results of the theory in Section \ref{sec2} below.

For $t\geq 0$, put $\eta^{{\rm jump}}_t:=\eta P^0_t$, where the measure $\eta$ is as defined in Theorem B, and $P^0_t$ is the semigroup of $U_\alpha$ killed at $0$. For $x\in\mathbb{R}$, denote by $\bar P^x$ the semigroup of $U_\alpha$ stopped at $0$. It can be checked that $\bar P^\eta:=\int_\mbR P^x \eta({\rm d}x)$ is the excursion measure of a skew $\alpha$-stable L\'{e}vy process. Denote by $(\eta^{{\rm c}}_t)_{t>0}$ and $\hat P^{U_\alpha}$ the entrance law of $U_\alpha$ and the corresponding excursion measure, respectively.

\noindent {\bf Theorem C.}
\textit{Assume that the function $x\mapsto \Pb(|\zeta|>x)$ is regularly varying at $+\infty$ of index $-\beta$, $\beta\in (0,\alpha-1)$ and \eqref{eq:c+-} holds. Let $p\in[0,1]$ and $X$ be a Feller process having the entrance law $(p \eta^{{\rm jump}}_t+(1-p) \eta^{{\rm c}}_t)_{t>0}$ and the corresponding excursion measure $p \bar P^\eta+(1-p) \hat P^{U_\alpha}$, where the measure $\eta$ is defined in \eqref{eq:Levy measure}. Then $X$ is a weak solution to the SDE
\bel{eq:SKEW_L1}
X(t)=X(0)+U_\alpha(t)+ p^{\frac{1}{\beta}}S_\beta(L^X_0(t)),\quad t\geq 0,
\ee
where $L^X_0$, $S_\beta$ are as given in Theorem B.
}

\textit{
Furthermore, the process  $X$ has a zero sojourn at $0$ with probability $1$.
}

\noindent {\bf Theorem D.}
\textit{Assume that the function $x\mapsto \Pb(|\zeta|>x)$ is regularly varying at $+\infty$ of index $-\beta$, $\beta\in (0,\alpha-1)$ and \eqref{eq:c+-} holds. Let $p\in[0,1]$ and $S_\beta$ be a $\beta$-stable L\'{e}vy process as defined in Theorem B. Then there exists a unique Feller process $X$ with a zero sojourn at $0$, which is a weak solution to equation \eqref{eq:SKEW_L1}, with  $U_\alpha$ being an $(\cF_t^X)_{t\geq 0}$-adapted process.} 
\begin{remk}
Even though we require in Theorem D $(\cF_t^X)_{t\geq 0}$-adaptedness of $U_\alpha$, it may follow automatically from  \eqref{eq:SKEW_L1} and independence of $U_\alpha$ and $S_\beta$.
\end{remk}

We close the section with the list of open problems which are non-trivial even for the SBM.

\noindent 1) Characterize a sticky skew $\alpha$-stable L\'{e}vy process and the corresponding SDE in the way similar to that used for the sticky Brownian motion, see \cite{EngelbertPeskir}.


\noindent 2) Consider a time inhomogeneous analogue of a skew stable L\'{e}vy process and describe the corresponding semigroup in terms of partial differential equations with Feller-Wenzell boundary condition at $0$, see \cite{kopytko2020} for the Brownian case.

\noindent 3) Investigate existence and uniqueness of a strong solution or a path-by-path solution in the sense of Davie \cite{Davie} to equation \eqref{eq:SKEW_L1}. This problem is non-trivial even in the case where the noise is a Brownian motion and the process $S_\beta$ is a subordinator \cite{Pilipenko:SkorokhodMap}, see also \cite{PilipekoPrykhodko2014jump_exit, IksanovPilipenkoPrykhodko2021}.

\noindent 4) Prove uniqueness of a weak solution to \eqref{eq:SKEW_L1} among all (possibly non-Markov) solutions in the situations that the local time is defined in terms of the time spent in a neighborhood of $0$, or a number of long excursions, or a number of a level crossings, etc.

The remainder of the paper is structured as follows. In Section \ref{sec1} we use a resolvent technique and prove Theorem A. In Section \ref{sec2} we recall some basic facts of It\^{o}'s excursion theory. In Section \ref{sec:SkewProofs} we prove Theorems B, C and D and their generalizations.

\section{Convergence of resolvents }\label{sec1}

\subsection{Discussion and limit theorem}

In view of the assumption $X_\zeta(0)=x\neq0$, the process $X_\zeta$ does not visit $0$.
It follows from the construction that $X_\zeta$ is a strong Markov process on $\mbR\setminus\{0\}$
with c\`{a}dl\`{a}g paths. We shall investigate distributional convergence of the processes $X_{\zeta/n}$ as $n\to\infty$.
Since we expect that a limit process visits $0$, the machinery of Markov processes on $\mbR$ rather than $\mbR\setminus\{0\}$ has to be exploited. To this end, we introduce an auxiliary {\it holding and jumping process} that spends at $0$ a random period of time having an exponential distribution (exponential time, in short), then has jump $\zeta_k$ and afterwards behaves like $U_\alpha$ until the next visit to $0$. The evolution just described then iterates, and the successive exponential times at $0$ are independent and identically distributed.

Here is a formal construction. For $m>0$, denote by $\tau_1$, $\tau_2,\ldots$ independent random variables having the  exponential distribution of mean $1/m$. Assume that the sequences $(\zeta_k)_{k\geq 1}$ and $(\tau_k)_{k\geq 1}$ and the process $U_\alpha$ are independent.
Similarly to \eqref{eq:def_of_X}, put
$$\tilde\sigma_0=0,\quad \tilde\sigma_{k+1}:=\inf\{t>\tilde \sigma_k+\tau_k\ : \
\zeta_k + U_\alpha(t)-U_\alpha(\tilde \sigma_k+\tau_k)=0\},\quad k\geq 0$$
and then
$$X_{\zeta,m}(t):=0, \quad \text{for}~~ t\in[\tilde \sigma_k, \tilde \sigma_k+\tau_k),\ k\geq 0
$$ and
$$X_{\zeta,m}(t):=\zeta_k+ U_\alpha(t)-U_\alpha(\tilde \sigma_k+\tau_k),\quad \text{for}~~ t\in[\tilde \sigma_k+\tau_k,
 \tilde \sigma_{k+1}), \ k\geq 0.$$ The so defined $X_{\zeta, m}$ is a Feller process on $\mbR$ with c\`{a}dl\`{a}g paths. Unlike $X_\zeta$, the process  $X_{\zeta,m}$ visits $0$ and may start at $0$.

Recall {\it Slutsky's lemma}: if $(X_n)_{n\geq 1}$ and $(Y_n)_{n\geq 1}$ are sequences of random elements in a metric space (with metric ${\rm dist}$) which satisfy ${\rm dist}\,(X_n,Y_n)\pto 0$ as $n\to\infty$, and the elements $X_n$ converge in distribution as $n\to\infty$ to a random element $X$, then the elements $Y_n$ converge in distribution to $X$, too.

Note that
\[
d(X_{\zeta}, X_{\zeta,m})\pto 0,\quad m\to\infty,
\]
where $d$ is the $J_1$-metric on $D$ and $\overset{\Pb}\to$ denotes convergence in probability. Hence, distributional convergence of $X_{\zeta/n}$ to $X$ follows if we can show that $X_{\zeta/n, m_n}$ converges in distribution, where $(m_n)$ is a sequence which diverges to $+\infty$ sufficiently fast as $n\to\infty$.

For $x,y\in\mathbb{R}$ and a Markov process $X$, denote by $P_t(x,{\rm d}y)$ its transition probability function at time $t>0$. Also, for bounded measurable 
functions $f:\mathbb{R}\to\mathbb{R}$, we define the semigroup
\[
 \E^x f(X(t)):= \int_{\mbR} f(y) P_t(x, {\rm d}y)=\lg P_t, f\rg (x)=P_tf(x),\quad t\geq 0, \ x\in\mathbb{R}
\]
and the resolvent
\[
R_\lambda f(x):= R_\lambda^X f(x):=\E^x \int_0^\infty e^{-\lambda t} f(X(t)){\rm d}t = \int_0^\infty
 e^{-\lambda t} P_t f(x){\rm d}t,\quad x\in\mathbb{R}.
\]
Further, for $x,y\in\mathbb{R}$, denote by $P^0_t(x,{\rm d}y)$ the transition probability function at time $t>0$ for the process $X$ killed upon the first visit to $0$, that is,
\[
P^0_t(x, A)= \Pb^x(X(t)\in A, t<\sigma),\quad x\in\mathbb{R}
\]
for Borel sets $A$ on $\mathbb{R}$. Also, we define the semigroup of the killed process
\[
\E^x f(X(t))\1_{\{t<\sigma\}}= \int_{\mbR} f(y) P^0_t(x, {\rm d}y)=\lg P^0_t, f\rg (x)=P^0_t f(x),\quad t\geq 0, \ x\in\mathbb{R}
\]
and its resolvent
\[
V_\lambda f(x)= V_\lambda^X f(x):=\E^x \int_0^\sigma e^{-\lambda t} f(X(t)) {\rm d}t =
\int_0^\infty e^{-\lambda t} P^0_t f(x){\rm d}t,\quad x\in\mathbb{R}.
\]
In Section \ref{intro} we have used the same notation $V_\lambda$ for the resolvent of the particular killed Markov process $U_\alpha$. Hopefully, this does not lead to a confusion. For later use, we note that if $X$ is a strong Markov process, then \begin{multline}\label{eq:resolvent_general}
R_\lambda f(x)=\E^x \Big(\int_0^{\sigma(X)}+\int_{\sigma(X)}^\infty\Big) e^{-\lambda t} f(X(t)){\rm d}t\\ = V_\lambda f(x)+\E^x e^{-\lambda \sigma(X)} \int_0^\infty e^{-\lambda t} f(X(t+\sigma(X))){\rm d}t\\ = V_\lambda f(x)+\E^x e^{-\lambda \sigma(X)} R_\lambda f(0),\quad x\in\mathbb{R}.
\end{multline}

Theorem \ref{thm:weak_convergenceEK} states that the uniform convergence of resolvents entails distributional convergence of the corresponding Markov processes. We write $C_0(\mbR)$ for the space of all continuous functions vanishing at $\pm \infty$, equipped with the supremum norm.
\begin{thm}\label{thm:weak_convergenceEK}
Let $(R^{(n)}_\lambda)_{n\geq 1}$ be a sequence of resolvents of some Feller processes $(X_n)_{n\geq 1}$. Assume that, for each $f\in C_0(\mbR)$ and each $\lambda>0$,
\[
\lim_{n\to\infty}\sup_{x\in\mathbb{R}}|R^{(n)}_\lambda f(x)-R_\lambda f(x)|=0,
\]
where $R_\lambda$ is the resolvent of a Feller process $X$. Assume also that the variables $X_n(0)$ converge in distribution as $n\to\infty $ to $X(0)$. Then
\[
X_n\Rightarrow X,\quad n\to\infty
\]
on $D$.
\end{thm}

The proof follows from \cite[Theorem 2.5, Chapter 4]{Ethier+Kurtz:1986} and \cite[Theorem 4.2, Chapter 3]{Pazy}.

We shall write $R^{\zeta, m}_\lambda$ for the resolvent of $X_{\zeta, m}$. For bounded measurable functions $f:\mathbb{R}\to\mathbb{R}$ we intend to calculate $R^{\zeta, m}_\lambda f(0)$. Although this follows a standard pattern (see, for instance, \cite[pp.~136-137]{Blumenthal2012excursions}), we provide full details for completeness. Write
\begin{multline}\label{eq:resolvent_equ}
R^{\zeta, m}_\lambda f(0)=\E^0 \Big(\int_0^{\tau_1}+\int_{\tau_1}^\infty\Big) e^{-\lambda t} f(X_{\zeta, m}(t)){\rm d}t =
\lambda^{-1} \E(1-e^{-\lambda\tau_1}) f(0)\\+\E e^{-\lambda \tau_1}\E^\zeta \int_0^\infty e^{-\lambda t} f(X_{\zeta, m}(t)){\rm d}t= \frac{1}{m+\lambda} f(0)+ \frac{m}{m+\lambda}  \lg  P_\zeta, R^{\zeta,m}_\lambda f\rg.
\end{multline}
Using \eqref{eq:resolvent_general} we infer $$\lg  P_\zeta, R^{\zeta,m}_\lambda f\rg=\lg  P_\zeta, V^{U_\alpha}_\lambda f\rg+\E^\zeta e^{-\lambda \sigma(U_\alpha)} R_\lambda f(0)$$ having utilized the equality $\E^\zeta e^{-\lambda \sigma(X_{\zeta, m})}=\E^\zeta e^{-\lambda \sigma(U_\alpha)}$. Substituting this into \eqref{eq:resolvent_equ} and then solving for $R_\lambda f(0)$ yields
\begin{equation}\label{eq:resol}
\lambda R_\lambda^{\zeta,m} f(0)=
\frac{\frac{f(0)}{m}+\lg  P_\zeta,  V^{U_\alpha}_\lambda f\rg}{\frac{1}{m}+\lambda^{-1}
\E^\zeta(1-e^{-\lambda \sigma(U_\alpha)})}=
\frac{\frac{f(0)}{m}+\lg  P_\zeta, V^{U_\alpha}_\lambda f\rg}{\frac{1}{m}+\lg  P_\zeta, V^{U_\alpha}_\lambda 1\rg}.
\end{equation}

\begin{thm}\label{thm:convergence_resolvents}
Assume that either the function $x\mapsto \Pb(|\zeta|>x)$ is regularly varying at $+\infty$ of index $-\beta$, $\beta\in (0,1)$ and \eqref{eq:c+-} holds, or $\E|\zeta|<\infty$. Let $f:\mathbb{R}\to\mathbb{R}$ be any bounded measurable function.

\noindent (a) If $\beta<\alpha-1$, then
\bel{eq:1partTheoremResolvents}
\lim_{n\to\infty} \lambda R_\lambda^{\zeta/n, m_n} f(0)=
 \frac{\lg  \eta^\ast, V^{U_\alpha}_\lambda f\rg}{\lg  \eta^\ast, V^{U_\alpha}_\lambda 1\rg},
\ee
where $(m_n)_{n\geq 1}$ is any sequence of positive numbers satisfying $\lim_{n\to\infty}m_n\Pb(\zeta>n)=\infty$, and $\eta^\ast$ is the measure defined by \eqref{eq:measure_eta}.

\noindent (b) If $\beta>\alpha-1$ or $\E|\zeta|<\infty$, then
\[
\lim_{n\to\infty}   R_\lambda^{\zeta/n, m_n} f(0)=
 R^{U_\alpha}_\lambda f(0),
\]
where $(m_n)_{n\geq 1}$ is any sequence of positive numbers satisfying $\lim_{n\to\infty}m_n n^{1-\alpha}=\infty$, and $R^{U_\alpha}_\lambda$ is the resolvent of $U_\alpha$.
\end{thm}

We claim that
\[
\lim_{n\to\infty}\sup_{x\in\mathbb{R}}|R_\lambda^{\zeta/n, m_n}f(x)-R_\lambda f(x)|=0
\]
for any bounded measurable $f:\mathbb{R}\to\mathbb{R}$ and any $\lambda>0$. Here, for $x\in\mathbb{R}$,
\bel{eq:res_limit_proc}
R_\lambda f(x)=
\begin{cases}
V^{U_\alpha}_\lambda f(x)+\lambda^{-1}\E^x e^{-\lambda \sigma(U_\alpha)} \frac{\lg  \eta^\ast, V^{U_\alpha}_\lambda f\rg}{\lg  \eta^\ast, V^{U_\alpha}_\lambda 1\rg},\quad \text{if} \ \beta<\alpha-1,\\
R^{U_\alpha}_\lambda f(x),\quad \text{if} \ \beta>\alpha-1~\text{or}~\E|\zeta|<\infty.
\end{cases}
\ee

To check this, observe that $V^{U_\alpha}_\lambda=V^{X_{\zeta/n, m_n}}_\lambda$, $\E^x e^{-\lambda \sigma(X_{\zeta, m})}=\E^x e^{-\lambda \sigma(U_\alpha)}$ for $x\in\mathbb{R}$. Further, 
$R_\lambda f(0)=\lambda^{-1}\frac{\lg  \eta^\ast, V^{U_\alpha}_\lambda f\rg}{\lg  \eta^\ast, V^{U_\alpha}_\lambda 1\rg}$ if $\beta<\alpha-1$ and $R_\lambda f(0)=R^{U_\alpha}_\lambda f(0)$ if $\beta>\alpha-1$ or $\E|\zeta|<\infty$. With these at hand, invoking \eqref{eq:resolvent_general} yields, for any $x\in\mathbb{R}$,
\begin{multline*}
|R_\lambda^{\zeta/n, m_n}f(x)-R_\lambda f(x)|=\E^x e^{-\lambda \sigma(U_\alpha)}|R_\lambda^{\zeta/n, m_n}f(0)-R_\lambda f(0)|\\
\leq |R_\lambda^{\zeta/n, m_n}f(0)-R_\lambda f(0)| ~\to~0,\quad n\to\infty.
\end{multline*}
We have used Theorem \ref{thm:convergence_resolvents} for the limit relation.

According to \cite[Chapter V, \S 2, Theorem 2.8]{Blumenthal2012excursions}, $R_\lambda$ is the resolvent of a strongly continuous probability semigroup and, as such, the resolvent of a Feller process. This observation in combination with Theorem \ref{thm:weak_convergenceEK} leads to the following.
\begin{corl}\label{corl:conv_proc}
Assume that the variables $X_{\zeta/n}(0)$ converge in distribution as $n\to\infty$ to a random variable $\xi$. Then the processes $X_{\zeta/n}$ converge in distribution on $D$ to a Feller process $X$ with the resolvent $R_\lambda$ defined in \eqref{eq:res_limit_proc} and $X(0)$ having the same distribution as $\xi$. In particular, if $\beta>\alpha-1$, then the limit process is $\xi+U_\alpha$.
\end{corl}

Theorem A is an immediate consequence of Corollary \ref{corl:conv_proc}.

\subsection{Auxiliary results}

In this section we prove a few preparatory results needed for the proof of Theorem \ref{thm:convergence_resolvents}. We shall treat the cases $\beta<\alpha-1$ and $\beta>\alpha-1$ or $\E|\zeta|<\infty$ in slightly different ways. As a consequence, we provide two collections of auxiliary results designed to deal with these cases. Throughout this section we write $\sigma$ for $\sigma(U_\alpha)$, $R_\lambda$ for $R^{U_\alpha}_\lambda$ and $V_\lambda$ for $V^{U_\alpha}_\lambda$.

\noindent {\sc Auxiliary results for the case $\beta<\alpha-1$}. For $\lambda>0$, the density $u_\lambda$ of the resolvent   kernel of $U_\alpha$ satisfies $u_\lambda(x,y)=u_\lambda(y-x)$, $x,y\in\mathbb{R}$, where 
\[
u_\lambda(x)=\frac{1}{ \pi} \int_0^\infty\frac{\cos(x\theta)}{\lambda+\theta^\alpha} {\rm d}\theta,\quad x\in\mathbb{R}.
\]
Lemma \ref{lem:calculating_integrals} collects a couple of formulae to be used in what follows.
\begin{lem}\label{lem:calculating_integrals}
Let $\alpha\in (1,2)$.

\noindent (a) For $\gamma\in [0,\alpha-1)$ and $\lambda>0$,
\[
\int_0^\infty \frac{\theta^\gamma}{\lambda+\theta^\alpha}{\rm d}\theta=\frac{\Gamma(1-\frac{\gamma+1}{\alpha})\Gamma(\frac{\gamma+1}{\alpha})}{\alpha\lambda^{1-\frac{\gamma+1}{\alpha}}}
=
\frac{\pi}{\alpha  \sin\frac{\pi(\gamma+1)}{\alpha}} \frac{1}{\lambda^{1-\frac{\gamma+1}{\alpha}}}.
\]
In particular,
\[
u_\lambda(0)=
\frac{1}{\alpha  \sin\frac{\pi}{\alpha} } \frac{1}{\lambda^{1-\frac{\gamma+1}{\alpha}}}.
\]

\noindent (b) For $x\in\mathbb{R}$,
$$\int_0^\infty \frac{1- \cos(xy)}{y^\alpha} {\rm d}y = |x|^{\alpha-1}\frac{\Gamma(2-\alpha)\sin \frac{\pi\alpha}{2}}{\alpha -1}.$$
\end{lem}
\begin{proof}
While the first equality in the first formula of part (a) is a consequence of \cite[formula (3.241)(2)]{Gradshteyn+Ryzhik:2007}, the second equality follows from Euler's reflection formula $\Gamma(1-z)\Gamma(z)=\frac{\pi}{\sin (\pi z)}$ which holds true for any noninteger $z$. The second equality of part (a) is implied by the first with $\gamma=0$ and the formula $u_\lambda(0)=\pi^{-1}\int_0^\infty (\lambda+\theta^\alpha)^{-1}{\rm d}\theta$. Part (b) follows from \cite[formula (14.18)]{Sato}.
\end{proof}

While formula \eqref{eq:hit} of Lemma \ref{lem:asym} will be used in the proof of both parts of Theorem \ref{thm:convergence_resolvents}, formula \eqref{eq:592} will be used in the proof of Theorem \ref{thm:convergence_resolvents}(b).
\begin{lem}\label{lem:asym}
For $\alpha\in (1,2)$ and $\lambda>0$,
\begin{equation}\label{eq:hit}
\lambda V_\lambda 1(x)=\E^x(1- e^{-\lambda \sigma})~\sim~A_{\lambda, \alpha}|x|^{\alpha-1},\quad x\to 0,
\end{equation}
where $A_{\lambda, \alpha}:=\frac{\alpha\Gamma(2-\alpha)\sin \pi\alpha\sin\frac{\pi\alpha}2}{\pi(\alpha-1)}\lambda^{1-\frac{1}{\alpha}}$,
and
\bel{eq:592}
\Pb^1(\sigma>y)\sim B_\alpha y^{-1+\frac{1}{\alpha}},\quad y\to\infty,
\ee
where $B_\alpha:=\frac{\sin \pi\alpha \, \sin{\frac{\pi\alpha}2} \Gamma(1-\alpha)}{\pi \, \Gamma(1-\frac1\alpha)}$
\end{lem}
\begin{proof}
We start with proving \eqref{eq:hit}. According to \cite[Corollary II.5.8]{Bertoin:1996},
\[
\E^x e^{-\lambda \sigma}= \frac{u_\lambda(-x)}{u_\lambda(0)},\quad x\in\mathbb{R},
\]
whence
\begin{multline}\label{eq:asymptot_hitting}
u_\lambda(0) \E^x(1- e^{-\lambda \sigma})=\frac{1}{ \pi } \int_0^\infty\frac{1-\cos(x\theta)}{\lambda+|\theta|^\alpha} {\rm d}\theta=\\ \frac{|x|^{\alpha-1}}{ \pi } \int_0^\infty\frac{1- \cos(y)}{\lambda |x|^\alpha+y^\alpha}{\rm d}y \sim
\frac{|x|^{\alpha-1}}{ \pi } \int_0^\infty\frac{1- \cos(y)}{y^\alpha}{\rm d}y,\quad x\to 0.
\end{multline}
Invoking Lemma \ref{lem:calculating_integrals} we arrive at \eqref{eq:hit}.
Using the first equality in \eqref{eq:asymptot_hitting} and Lemma \ref{lem:calculating_integrals} we infer
\begin{multline*}
\E^1(1-e^{-\lambda\sigma})= \frac{\alpha\sin{\frac\pi\alpha}\lambda^{1-\frac1\alpha}}{\pi} \int_0^\infty \frac{1-\cos \theta}{\lambda+\theta^\alpha}{\rm d}\theta~\sim~\frac{\alpha\sin{\frac\pi\alpha}}{\pi} \int_0^\infty \frac{1-\cos \theta}{\theta^\alpha}{\rm d}\theta\, \lambda^{1-\frac1\alpha}\\= \frac{\sin \pi\alpha \, \sin{\frac{\pi\alpha}2} \Gamma(1-\alpha)}{\pi} \, \lambda^{1-\frac1\alpha},\quad \lambda\to 0+.
\end{multline*}
An application of Corollary 8.1.7 in \cite{BTG} yields
\eqref{eq:592}.
\end{proof}

Lemma \eqref{eq:reg} is the principal ingredient of the proof of Theorem \ref{thm:convergence_resolvents}(a). 

\begin{lem}\label{eq:reg}
Let $\alpha\in(1,2)$ and $g:\mathbb{R}\to\mathbb{R}$ be a bounded measurable 
function satisfying $g(x)=O(|x|^{\alpha-1})$ as $x\to0$. Assume that the function $x\mapsto \Pb(|\zeta|>x)$ is regularly varying at $+\infty$ of index $-\beta$, $\beta\in (0,\alpha-1)$ and \eqref{eq:c+-} holds. Then
\[
\lim_{n\to\infty}  \frac{\E g(\zeta/n)}{\Pb(\zeta>n)}=\int_{\mathbb{R}}g(x)\eta^\ast({\rm d}x)=\lg \eta^\ast, g \rg<\infty,
\]
where $\eta^\ast$ is the measure defined in \eqref{eq:measure_eta}. 
\end{lem}
\begin{proof}
The functions $g^+:=\max (g,0)$ and $g^-:=\max (-g,0)$ are nonnegative, bounded and satisfy $g^\pm(x)=O(|x|^{\alpha-1})$ as $x\to 0$. Thus, without loss of generality we can and do assume that $g$ is nonnegative.

Put $G(x):=\Pb(\zeta\leq x)$ for $x\in\mathbb{R}$. We shall show that
\begin{equation}\label{eq:half1}
\lim_{n\to\infty}  \frac{\int_{[0,\,\infty)}g(x){\rm d}G(nx)}{1-G(n)}=c_+\beta \int_0^\infty g(x)x^{-\beta-1}{\rm d}x.
\end{equation}
Fix any $\ve \in(0,1)$. Then
\[
\lim_{n\to\infty}  \frac{\int_{ (\ve,\,\infty)}g(x){\rm d}G(nx)}{1-  G(n)}= c_+\beta \int_\ve^\infty g(x) x^{-\beta-1}{\rm d}x
\]
follows from $$\lim_{n\to\infty}\frac{\Pb(\zeta>nx)}{\Pb(\zeta>n)}=c_+x^{-\beta},\quad x>0.$$ Observe that the usual requirement of continuity of $g$ is not needed here, for the measure $\eta^\ast$ is (absolutely) continuous.

There is a constant $c>0$ such that $g(x)\leq c x^{\alpha-1}$ whenever $x\in (0,1]$. With this at hand we conclude that
\[
\int_{[0,\,\ve]} g(x){\rm d}G(nx)\leq \int_{[0,\,\ve]} c x^{\alpha-1}{\rm d}G(nx)=\frac{c}{n^{\alpha-1}}\int_{[0,\,n\ve]} x^{\alpha-1}{\rm d}G(x).
\]
Further,
\[
\int_{[0,\,n\ve]} x^{\alpha-1} {\rm d} G(x)~\sim~(n\ve)^{\alpha-1}(1-G(n\ve))\frac{\beta}{\alpha-1-\beta}~\sim~ \frac{\beta}{\alpha-1-\beta}\ve^{\alpha-1-\beta}n^{\alpha-1}(1-G(n))
\]
as $n\to\infty$,
where the first asymptotic relation follows from Karamata's theorem \cite[Theorem 1.6.4]{BTG}. We infer
\[
\limsup_{n\to\infty}\int_{[0,\,\ve]} g(x) \frac{{\rm d} G(nx)}{1-G(n)}\leq
  c \frac{\beta}{\alpha-1-\beta}\ve^{\alpha-1-\beta}
\]
and
\[
\limsup_{n\to\infty}\int_{[0,\,\infty)} g(x) \frac{{\rm d} G(nx)}{1-G(n)}\leq
  c \frac{\beta}{\alpha-1-\beta}\ve^{\alpha-1-\beta}+ c_+\beta \int_\ve^\infty g(x) x^{-\beta-1} {\rm d}x.
\]
Sending $\ve\to0+$ we arrive at
\[
\limsup_{n\to\infty}\int_{[0,\infty)} g(x) \frac{{\rm d} G(nx)}{1-G(n)}\leq
    c_+\beta \int_0^\infty g(x) x^{-\beta-1}{\rm d}x.
\]
For the lower bound, write, for any $\ve>0$,
\[
\int_{[0,\,\infty)} g(x) \frac{{\rm d} G(nx)}{1-G(n)}\geq
\int_{(\ve,\,\infty)} g(x) \frac{{\rm d} G(nx)}{1-G(n)}~\to~
   c_+\beta \int_\ve^\infty g(x) x^{-\beta-1}{\rm d}x,\quad n\to\infty.
\]
Sending $\ve \to0+$ we obtain
\[
\liminf_{n\to\infty}\int_{[0,\,\infty)} g(x) \frac{{\rm d} G(nx)}{1-G(n)}\geq
   c_+\beta \int_0^\infty g(x) x^{-\beta-1} {\rm d}x,
\]
and \eqref{eq:half1} follows.

Starting with $$\lim_{n\to\infty}\frac{\Pb(-\zeta>nx)}{\Pb(\zeta>n)}=c_-x^{-\beta},\quad x>0$$ and arguing analogously we also conclude that
$$\int_{(-\infty,\, 0)} g(x) \frac{{\rm d} G(nx)}{1-G(n)}=c_-\beta \int_{-\infty}^0 g(x) |x|^{-\beta-1} {\rm d}x.$$ Combining this with \eqref{eq:half1} completes the proof of the lemma.
\end{proof}

\noindent {\sc Auxiliary results for the case $\beta>\alpha-1$ or $\E|\zeta|<\infty$}.
Lemmas \ref{lem:360} and \ref{lem:361} will be used for the proof of Theorem \ref{thm:convergence_resolvents} (b).
\begin{lem}\label{lem:360}
For $\lambda>0$ and any bounded and continuous function $f:\mathbb{R}\to\mathbb{R}$,
\bel{eq:limit_resolvents}
\lim_{x\to 0}\frac{V_\lambda f(x)}{V_\lambda 1(x)}=
\lambda R^{U_\alpha}_\lambda f(0)=\lambda \int_\mbR u_\lambda(y) f(y){\rm d}y.
\ee
\end{lem}
\begin{remk}\label{remk:360}
According to this lemma and \cite[Theorem 4.2, Section V.4]{Blumenthal2012excursions} there exists a unique recurrent extension of the process $U_\alpha$ killed at $0$ which has a zero sojourn at $0$. As a consequence, this extension coincides with $U_\alpha$. More details on the excursion theory, recurrent extensions of Markov processes etc.\ will be given in Section  \ref{sec2}.
\end{remk}
\begin{proof}
For each $x\in\mathbb{R}$, denote by $X_x$ the process defined in \eqref{eq:def_of_X} in which we formally replace $\zeta$ with $x$. For $k\in \mbN$, denote by $\sigma_k^{(x)}$ the time of the $k$th jump (of size $x$) from $0$ and
note that the random variables $\sigma_1^{(x)}$, $\sigma_2^{(x)},\ldots$ are independent and identically distributed. We shall use a representation
\[
X_x(t)=U_\alpha(t)+ x N_\alpha(x,t),\quad t\geq 0,~~x\in\mathbb{R},
\]
where $N_\alpha(x,t)=\#\{k\in\mbN\ :\ \sigma_1^{(x)}+\ldots+ \sigma_k^{(x)}\leq t\}$.

%

Arguing as in \eqref{eq:resolvent_general} and \eqref{eq:resolvent_equ} we conclude that
\[
\frac{V_\lambda f(x)}{\lambda V_\lambda 1(x)}= \E\int_0^\infty e^{-\lambda t}
f(X_x(t)){\rm d}t,\quad x\in\mathbb{R}.
\]
In view of this and the Lebesgue dominated convergence theorem, we are left with showing that
\bel{eq:765}
x N_\alpha(x,t)\overset{\Pb}\to 0,\quad x\to 0.
\ee

To prove \eqref{eq:765}, recall that the process $U_\alpha$ is self-similar of index $1/\alpha$. This implies that, for $k\in \mbN$, $\sigma^{(x)}_k$ has the same distribution as $\sigma^{(1)}_k |x|^\alpha$ and thereupon $N_\alpha(x,t)$ has the same distribution as $\#\{k\in\mbN\ : \ \sigma_1^{(1)}+\ldots+ \sigma_k^{(1)}\leq t|x|^{-\alpha}\}$. By \cite[Theorem 3.2]{MeerschaertShceffler2004},
\[
\Pb^1(\sigma>t|x|^{-\alpha}) \, N_\alpha(x,t)\overset{{\rm d}}\to S_{1-\frac{1}{\alpha}}^\leftarrow (1):=\inf\{ t\geq 0\ :\  S_{1-\frac{1}{\alpha}}(t)>1\},\quad x\to 0,
\]
where $\overset{{\rm d}}\to$ denotes convergence of one-dimensional distributions, and $( S_{1-\frac{1}{\alpha}}(t))_{t\geq 0}$ is a drift-free $(1-\frac{1}{\alpha})$-stable subordinator with
\[
-\log \E e^{-u S_{1-\frac{1}{\alpha}}(1)}=\Gamma(1/\alpha) u^{1-\frac{1}{\alpha}},\quad u\geq 0.
\]
Since, by \eqref{eq:592}, $P^1(\sigma>t|x|^{-\alpha})\sim B_\alpha t^{-1+\frac1\alpha}|x|^{\alpha-1}$ as $x\to 0$, we infer \eqref{eq:765}.
\end{proof}
\begin{lem}\label{lem:361}
%
Let $a:\mathbb{R}\to [0,\infty)$ and $b:\mathbb{R}\to\mathbb{R}$ be functions such that $|b(x)|\leq C a(x)$ for all $x\in\mbR$ and some constant $C>0$ and that $\lim_{x\to0}\frac{b(x)}{a(x)}=A\in \mathbb{R}$. Assume that a sequence $(\mu_n)_{n\geq 1}$ of probability measures satisfies $\mu_n(\{0\})=0$ for large $n$ and, for each $\delta>0$,
\bel{eq:367}
 \lim_{n\to\infty}\frac{\int_{|x|\geq \delta} a(x)\mu_n({\rm d}x)}{\int_{|x|< \delta} a(x)\mu_n({\rm d}x)}=0.
\ee
Then
\[
 \lim_{n\to\infty}\frac{\int_{\mbR} b(x)\mu_n({\rm d}x)}{\int_{\mbR} a(x)\mu_n({\rm d}x)}=A.
\]
\end{lem}
\begin{proof}
Let $\ve>0$ be arbitrary. Select $\delta>0$ such that $|\frac{b(x)}{a(x)}-A|<\ve$ whenever $|x|\leq \delta$.
Then
\begin{multline*}
\limsup_{n\to\infty}\frac{\int_{\mbR} b(x)\mu_n({\rm d}x)}{\int_{\mbR} a(x)\mu_n({\rm d}x)}=\limsup_{n\to\infty}
\frac{\frac{(\int_{|x|\geq \delta}+\int_{|x|< \delta}) b(x)\mu_n({\rm d}x)}{\int_{|x|<\delta} a(x)\mu_n({\rm d}x)}}{\frac{
(\int_{|x|\geq \delta}+\int_{|x|< \delta})
 a(x)\mu_n({\rm d}x)}{\int_{|x|<\delta} a(x)\mu_n({\rm d}x)}}=\limsup_{n\to\infty}
\frac{\frac{\int_{|x|< \delta} b(x)\mu_n({\rm d}x)}{\int_{|x|<\delta} a(x)\mu_n({\rm d}x)}}{\frac{
\int_{|x|< \delta}
 a(x)\mu_n({\rm d}x)}{\int_{|x|<\delta} a(x)\mu_n({\rm d}x)}}=
\\
\limsup_{n\to\infty} \frac{\int_{|x|<\delta} b(x)\mu_n({\rm d}x)}{\int_{|x|<\delta} a(x)\mu_n({\rm d}x)}
\leq \limsup_{n\to\infty} \frac{(A+\ve)\int_{|x|<\delta} a(x)\mu_n({\rm d}x)}{\int_{|x|<\delta} a(x)\mu_n({\rm d}x)}= A+\ve
\end{multline*}
having utilized for the second equality $$\lim_{n\to\infty}\frac{\int_{|x|\geq \delta} b(x)\mu_n({\rm d}x)}{\int_{|x|< \delta} a(x)\mu_n({\rm d}x)}=0$$ which is a consequence of \eqref{eq:367} and the assumption $|b(x)|\leq Ca(x)$ for all $x\in\mathbb{R}$.
An analogous inequality for the lower limit follows similarly. Sending $\ve\to 0$ completes the proof of the lemma.
\end{proof}

\subsection{Proof of Theorem \ref{thm:convergence_resolvents}}

We first prove the theorem in the case $\beta<\alpha-1$ (part (a) of the theorem). In view of \eqref{eq:resol} we have to show that
\begin{equation}\label{eq:inter}
\lim_{n\to\infty}\frac{f(0)/m_n+\lg  P_{\zeta/n}, V_\lambda f\rg}{1/m_n+\lg  P_{\zeta/n}, V_\lambda 1\rg}=\frac{\lg  \eta^\ast, V_\lambda f\rg}{\lg  \eta^\ast, V_\lambda 1\rg}.
\end{equation}

The function $V_\lambda 1$ is nonnegative and bounded (by $\lambda^{-1}$). According to \eqref{eq:hit}, it satisfies $V_\lambda 1(x)=O(|x|^{\alpha-1})$ as $x\to0$. By virtue of
\begin{equation}\label{eq:bound}
|V_\lambda f(x)|\leq \sup_{y\in\mathbb{R}}|f(y)| V_\lambda 1(x),\quad x\in\mathbb{R}
\end{equation}
and the fact that $f$ is bounded by assumption, we conclude that $V_\lambda f$ is a bounded function satisfying $V_\lambda f(x)=O(|x|^{\alpha-1})$ as $x\to0$. Hence, an application of Lemma \eqref{eq:reg} yields
$$\lg  P_{\zeta/n}, V_\lambda f \rg ~\sim~ \Pb(\zeta>n)\lg  \eta^\ast, V_\lambda f\rg\quad\text{and}\quad \lg  P_{\zeta/n}, V_\lambda 1 \rg~\sim~\Pb(\zeta>n)\lg  \eta^\ast, V_\lambda 1\rg,\quad n\to\infty.$$ These together with our choice of $m_n$ prove \eqref{eq:inter}.

Now we are turning to the proof of part (b). Formula \eqref{eq:resol} tells us that we have to prove that
\begin{equation}\label{eq:inter1}
\lim_{n\to\infty}\frac{f(0)/m_n+\lg  P_{\zeta/n}, V_\lambda f\rg}{1/m_n+\lg  P_{\zeta/n}, V_\lambda 1\rg}=\lambda R_\lambda^{U_\alpha} f(0).
\end{equation}

In the notation of Lemma \ref{lem:361}, put $a:=V_\lambda 1$, $b:=V_\lambda f$ and $\mu_n:=P_{\zeta/n}$ for $n\geq 1$. Then, in view of \eqref{eq:bound}, we may put $C:=\sup_{y\in\mathbb{R}}|f(y)|$ and, by Lemma \ref{lem:360}, $A:=\lambda R^{U_\alpha}_\lambda f(0)$. Now relation \eqref{eq:inter1} follows from Lemma \ref{lem:361} and our choice of $m_n$ if we can show that
\begin{equation}\label{eq:inter100}
\lim_{n\to\infty}m_n \lg  P_{\zeta/n}, V_\lambda 1\rg=\infty
\end{equation}
and that, for any $\delta>0$,
\begin{equation}\label{eq:367a}
\lim_{n\to\infty}\frac{\int_{|x|\geq \delta} \E^x (1-e^{-\lambda \sigma})P_{\zeta/n}({\rm d}x)}{\int_{|x|< \delta} \E^x (1-e^{-\lambda \sigma})P_{\zeta/n}({\rm d}x)}=0
\end{equation}
which is condition \eqref{eq:367} with $a(x)=V_\lambda 1(x)=\lambda^{-1}\E^x (1-e^{-\lambda \sigma})$ for $x\in\mathbb{R}$.

Fix any $\delta>0$. On the one hand, $$\int_{|x|\geq \delta} \E^x (1-e^{-\lambda \sigma})P_{\zeta/n}({\rm d}x)\leq \Pb(|\zeta|\geq n\delta)=o(n^{1-\alpha}),\quad n\to\infty$$ which holds true if $\E|\zeta|<\infty$ or $\beta>\alpha-1$. On the other hand, appealing to \eqref{eq:hit} we conclude that, for appropriate constant $K>0$,
\begin{multline}\label{eq:1000}
\int_{|x|<\delta} \E^x(1- e^{-\lambda \sigma})P_{\zeta/n}({\rm d}x)\geq K \int_{|x|<\delta} |x|^{\alpha-1} P_{\zeta/n}({\rm }x)= K n^{1-\alpha }\int_{|x|<n\delta} |x|^{\alpha-1} \Pb_\zeta({\rm d}x)\\~\sim~ K \E(|\zeta|^{\alpha-1})n^{1-\alpha},\quad n\to\infty.
\end{multline}
Since $\E(|\zeta|^{\alpha-1})<\infty$ if $\E|\zeta|<\infty$ or $\beta>\alpha-1$, \eqref{eq:367a} follows. Finally, \eqref{eq:inter100} is a consequence of \eqref{eq:1000} and our assumption $\lim_{n\to\infty}m_n n^{1-\alpha}=\infty$.

The proof of Theorem \ref{thm:convergence_resolvents} is complete.

\section{It\^{o}'s excursion theory} \label{sec2}

According to Corollary \ref{corl:conv_proc}, the processes $X_{\zeta/n}$ converge in distribution. If $\beta>\alpha-1$ or $\E|\zeta|<\infty$, the limit process is $U_\alpha$ and, as such, well-understood. If $\beta<\alpha-1$, we only have a resolvent description of the
limit process that we call a skew $\alpha$-stable L\'{e}vy process. In this section, we give a probabilistic representation of this process via It\^{o}'s excursion theory and an SDE, which include a local time of the process.  There are several definitions of a local time of a Markov process. Usually these produce the same process, up to a multiplicative constant. To provide an absolutely rigorous formulation of our results, we should stick to a particular definition that serves our needs. As far as the skew stable L\'{e}vy process is concerned, the exact values of constants are as important as the parameter of permeability appearing in the defining equation \eqref{eq:SBM} for the SBM.

We briefly review below some basic facts of It\^{o}'s excursion theory for Markov processes. We only discuss real-valued processes and leave aside processes taking values in the other spaces. We follow Blumenthal's book \cite{Blumenthal2012excursions} and cite specific theorems or point out pages for the most  important results. Let $(X, P^x)$, $x\in\mathbb{R}$ be a Feller process, that is, a Markov process whose transition semigroup $(P_t)_{t\geq 0}$ is strongly continuous on $C_0(\mbR)$. Without loss of generality we always assume that all processes to be considered are standard, see the corresponding definition in \cite[Chapter I, \S 9]{BlumentalGetoor}.
 Assume that $X$ is recurrent at $0$, that is, $P^x(\sigma<\infty)=1$ for all $x\in\mbR$, where, as before, $\sigma=\sigma(X)$ denotes the first hitting time of $0$.

%

Associated with $(X, P^x)$ are

\noindent (i) $\bar X(t):=X(t\wedge \sigma)$, $t\geq 0$ the process stopped at $0$ (we denote by $\bar P^x$ and $\bar P_t$ its distribution and semigroup, respectively), and

\noindent (ii) the process killed at $0$ with the semigroup $P^0_t$
and the transition probabilities $P^0(t,x,A)=P(X(t)\in A, \ t<\sigma)$ for $t\geq 0$ and $x\in\mathbb{R}$.

\noindent The killed process and its semigroup will be called {\it minimal process} and {\it minimal semigroup}, respectively.

Below we recall elements of It\^{o}'s synthesis theory, which describes all recurrent extensions of the minimal process.


Assume that

\textit{ $$\text{the function x}~\mapsto \bar\E^x e^{-\lambda\sigma}~\text{is continuous for each}~ \lambda>0~\text{and} \lim_{|x|\to\infty} \bar \E^x e^{-\sigma}=0.$$}

%
%

\begin{remk}
All the properties of a minimal semigroup, stated above and below,
are satisfied for the minimal semigroup
of a symmetric $\alpha$-stable L\'{e}vy process with $\alpha\in(1,2)$.
\end{remk}

To each function $u \in  D([0,\infty))$, we associate $\sigma(u):=\inf\{t>0 : u(t)=0\}$. Let $\hat \Pb$  be a $\sigma$-finite measure on $D$ supported by the set of functions $\{ u \ : \ u(t)=0, t\geq \sigma(u)\}$. The elements of this set will be called {\it excursions} and $\sigma(u)$ will be called the {\it length of excursion $u$}. We shall assume that
\[
\hat \Pb (1-e^{-\sigma})\leq 1
\]
and that $\hat P ( |u(0)|>x) <\infty$ for $x>0$.




Let $N({\rm d}s, {\rm d}u)$ denote a Poisson point measure on $[0,\infty)\times D$ with intensity ${\rm d}s\times \hat \Pb({\rm d}u)$. Denote by $(s_k, u_k)$ the atoms of $N$, that is, $N=\sum_k\delta_{(s_k, u_k)}$. Put $m:=1- \hat \Pb (1-e^{-\sigma})$,
\begin{equation*}
\tau(s):=ms+\sum_{s_k\leq s} \sigma(u_k)=
ms+\int_{[0,\,s]}\int_D \sigma(u) N({\rm d}z, {\rm d}u),\quad s\geq 0
\end{equation*}
and
\begin{equation}\label{eq:inverse}
\vf(t):= \inf\{s\geq 0\  :\ \tau(s)> t\},\quad t\geq 0.
\end{equation}
Assume that at least one of the following conditions holds:

\textit{ \centerline{ $m>0$ or the measure $\hat \Pb $ is infinite.}}
Then $(\tau(s))_{s\geq 0}$ is a strictly increasing subordinator. For  $t\in [\tau(s_k-), \tau(s_k)]$, put $X(t):= u_k(t-\tau(s_k-))$ and, for $t\notin \cup_k[\tau(s_k-), \tau(s_k)]$, put $X(t):=0$.


Further, we assume that the characteristic measure $\hat P$ is {\it compatible with the minimal semigroup}. This means that, for all $n\geq 1$, all $0\leq s_1\leq\ldots\leq s_n\leq s$ and all bounded measurable functions $g, g_1,\ldots,g_n$ with  $g(0)=0$,
\[
\hat P\Big(g(u(t+s))g_1(u(s_1))\ldots g_n(u(s_n)); \sigma>s\Big) =
 \hat P\Big(P^0_t g(u(s))g_1(u(s_1))\ldots g_n(u(s_n))
; \sigma>s\Big),
\]
and
\[
\hat P\Big(g(u(t)) ; u(0)\in B\Big) = \hat P\Big(P^0_t g(u(0)) ; u(0)\in B\Big),
\]
(see \cite[Chapter V \S 1 and V \S 2 (d)]{Blumenthal2012excursions} for more details). Then, according to \cite[Theorem 2.10, Chapter V \S2]{Blumenthal2012excursions}), $X$ is a Markov extension of the minimal process with $X(0)=0$. Moreover, each extension (starting at $0$) can be obtained with the help of merging excursions, under suitable measure $\hat P$ satisfying the compatibility conditions. An extension has a zero sojourn at $0$ if, and only if, $m=0$.

\begin{remk}
If $X(0)\neq 0$, then the behavior of $X$ until hitting $0$ is determined uniquely by the minimal semigroup or the
semigroup corresponding to the stopped process.
 \end{remk}
Note that the measure $\hat P$ is finite if, and only if, $\hat P=\bar P^\theta$ for some finite measure $\theta$ on $\mbR\setminus\{0\}$. In this case, $X$ is a holding and jumping process, with holding time $\|\theta\|$ and jumping distribution $\frac{\theta}{\|\theta\|}$, where $\|\theta\|=\theta(\mbR\setminus\{0\})$.

If the measure $\hat P$ is infinite, then $X$ can be obtained as an a.s.\ limit of holding and jumping processes.
To explain the construction, recall that \cite[Chapter 5.2]{Blumenthal2012excursions}
a family of $\sigma$-finite measures $(\theta_t)_{t>0}$ on the Borel subsets of $\mbR\setminus\{0\}$
is called {\it entrance law} for $P^0_t$, if
\[
\theta_s P_t^0 =\theta_{t+s}, t\geq 0,\quad s>0
\]
and
\[
\lg\theta_s, V 1\rg = \E^{\theta_s}(1-e^{-\sigma})\leq 1,\quad s>0.
\]
If $\theta$ is a $\sigma$-finite measure on $\mbR\setminus \{0\}$ such that
$\E^{\theta}(1-e^{-\sigma})\leq1$, then $\theta_t:=\theta P_t^0$ is an entrance law. Another important  example of an entrance law is $\hat \theta_t(A):=\hat \Pb(X_t\in A,\ t<\sigma)$.

Consider a family of holding and jumping processes $(X_\ve)_{\ve>0}$ with jumping distribution
$\frac{\hat\eta_\ve}{\|\hat\eta_\ve\|}$ and holding parameter
$m+\delta_\ve$, where $\delta_\ve>0$ is any function satisfying $\delta_\ve=o(\ve)$ as $\ve\to0$. It follows from the proof of \cite[Theorem 2.10, Chapter 5.2]{Blumenthal2012excursions} that the corresponding resolvents converge uniformly.
By Theorem \ref{thm:weak_convergenceEK}, this entails the distributional convergence $X_\ve\Rightarrow X$ on $D$ as $\ve\to0$. The aforementioned proof contains an explicit construction of $X_\ve$, which includes merging of excursions of the length larger than $\ve>0$, and such that
$
\lim_{\ve\to 0} X_\ve(t)= X(t)\quad \mbox{a.s.}
$


Note also that, for any entrance law $(\theta_t)$, there exists a unique characteristic measure $\hat P$ satisfying  $\theta_t=\hat \theta_t, t>0,$ see \cite[Theorem 4.7, Chapter V]{Blumenthal2012excursions}.
Moreover, the characteristic measure $\hat P$ can be recovered from the process $X$, see
\cite[Chapter III]{Blumenthal2012excursions}, and the process $(\vf(t))_{t\geq 0}$ defined in \eqref{eq:inverse} is the Blumenthal-Getoor local time of $X$ at $0$, see \cite[Theorem 2.3, Chapter V \S 2]{Blumenthal2012excursions}, that is, a continuous additive
functional which satisfies
\[
\E^x e^{-\sigma} = \E^x\int_0^\infty e^{-t} {\rm d}\vf(t),\quad x\in\mathbb{R}.
\]
Existence, uniqueness and some other properties of the local time are discussed in \cite[p. 91-93]{Blumenthal2012excursions}.
Therefore, under some natural properties imposed on the minimal semigroup, there is a one-to-one correspondence between the recurrent extension of the minimal process, the characteristic measure $\hat P$ and the entrance law $(\theta_t)$. We stress that the analysis of the entrance law is simpler than that of the two other objects.

Any entrance law for $P^0_t$ can be uniquely decomposed as the sum of two entrance laws
\[
\theta_t=\rho_t+ \theta P_t^0,\quad t>0,
\]
where $\theta$ is a $\sigma$-finite measure, and the measure
$\rho_t$ satisfies
\[
\lim_{t\to0+}\rho_t(\mbR\setminus [-x,x])=0
\]
for any $x>0,$ see \cite[p. 140]{Blumenthal2012excursions}.
These entrance laws, which are called {\it continuous entrance} and {\it jump entrance}, respectively, can also be characterized as follows:
\[
\theta_t^c(\cdot):=\rho_t=\hat P(X(s)\in \cdot,\ X(0)=0, \ s<\sigma),
\]
\[
\theta_t^j(\cdot): =\hat P(X(s)\in \cdot,\ X(0)\neq 0, \ s<\sigma),
\]
see \cite[p. 156]{Blumenthal2012excursions}. It follows from Theorems 2.6, 2.8 and the proof of Theorem 2.10 from \cite[Chapter V]{Blumenthal2012excursions} that if an entrance law is a jump entrance law, that is, $\theta_t
=\theta P_t^0,$ where $\theta$  is an infinite measure satisfying $\E^\theta(1-e^{-\sigma})=1$, then the corresponding recurrent extension $X$ has a zero sojourn at $0$, its resolvent satisfies
\bel{eq:Resolvent_at0}
\lambda R_\lambda f(0)=\frac{\lg  \theta, V_\lambda f\rg}{\lg  \theta, V_\lambda 1\rg},
\ee
the corresponding characteristic measure $\hat P$ is equal to $\bar P^\theta$.




Here is another conclusion motivated by the aforementioned facts. Assume that there exists a unique continuous
entrance law  $\rho_t$ with the characteristic measure $\hat P^c$ that  satisfies $\hat P^c(1-e^{-\sigma})=1$. For instance, this is the case for the process $U_\alpha$, see Remark \ref{remk:360}. Then any entrance law $(\theta_t)$, which corresponds to an extension with a zero sojourn at $0$, can be uniquely represented by
\bel{eq:entrance-general_decomposition}
 \theta_t=(1-p)\rho_t+ p\,\theta P_t^0,
 \ee
where $p\in[0,1]$, and the measure $\theta$ satisfies $\E^\theta(1-e^{-\sigma})=1$. The corresponding characteristic measure  is
\bel{eq:characteristic-general_decomposition}
\hat P= (1-p) \hat P^c+ p  \bar P^\theta.
\ee



\section{SDE for the skew stable L\'{e}vy  process}\label{sec:SkewProofs}

In this section all semigroups, resolvents, etc.\ are related to extensions of the minimal process corresponding to the process $U_\alpha$ killed at 0.

\subsection{Existence of solutions to SDEs with local times}\label{sec:Existence}
Let $X$ be a Feller process with the resolvent
\begin{equation}\label{eq:reso}
R_\lambda f(x)=V_\lambda f(x)+\E^x e^{-\lambda \sigma} \frac{\lg  \theta, V_\lambda f\rg}{\lg  \theta, V_\lambda 1\rg},\quad \lambda>0,~~x\in\mathbb{R},
\end{equation}
where $\theta$ is an infinite measure satisfying $\E^\theta(1-e^{-\sigma})=1$. The right-hand side in \eqref{eq:reso} is obtained by a combination of \eqref{eq:resolvent_general} and \eqref{eq:Resolvent_at0}.
\begin{thm}\label{thm:eq_Levy_1}
The process $X$ is a (weak) solution to the SDE
\bel{eq:SDE_Levy}
X(t) = X(0)+  U_\alpha(t) + S_\theta(L^X_0(t)),\quad t\geq 0,
\ee
where $U_\alpha$ is a symmetric $\alpha$-stable L\'{e}vy process, $S_\theta$ is a pure-jump L\'{e}vy process with the L\'{e}vy measure $\theta$, which is independent of $U_\alpha$, the process $L^X_0$ is the local time of $X$ at $0$ (see the definition in the previous subsection).
\end{thm}
%

A specialization of Theorem \ref{thm:eq_Levy_1}, with the L\'{e}vy measure $\theta=\eta$ given by \eqref{eq:Levy measure} and $S_\theta=S_\beta$, immediately proves Theorem B, except the explicit formula for the constant $C$. Here is the remaining piece of the proof. Write
\begin{multline*}
\frac{1}{C}=\int_{\mbR}\E^x( 1-e^{-\sigma(U_\alpha)})\eta({\rm d}x)= \frac{c_-+c_+}{\pi u_1(0)}\int_0^\infty
\int_0^\infty\frac{1-\cos(x\theta)}{1+\theta^\alpha} {\rm d}\theta \frac{{\rm d}x}{{x^{1+\beta}}}\\=\frac{c_-+c_+}{\pi u_1(0)}\frac{\Gamma(1-\beta)\sin\frac{\pi(1+\beta)}{2}}{\beta}\int_0^\infty\frac{\theta^\beta}{1+\theta^\alpha}{\rm d}\theta=
\frac{(c_-+c_+)\Gamma(1-\beta) \cos\frac{\pi\beta}{2}\sin\frac{\pi}{\alpha}}{\beta\sin\frac{\pi(\beta+1)}{\alpha}}
\end{multline*}
having utilized \eqref{eq:Levy measure} and \eqref{eq:asymptot_hitting} for the second equality, Fubini's theorem and Lemma \ref{lem:calculating_integrals}(b) for the third equality and Lemma \ref{lem:calculating_integrals}(a) for the last equality.

\begin{proof}[Proof of Theorem \ref{thm:eq_Levy_1}]
Since $\E^\theta (1-e^{-\sigma})=1$ by assumption, we conclude that the `holding' parameter $m=1-\E^\theta (1-e^{-\sigma})$ is equal to $0$, and the process $X$ has a zero sojourn at $0$ a.s.


Let $N_\theta({\rm d}s, {\rm d}x)$ be a Poisson point measure with intensity ${\rm d}s \times  \theta({\rm d}x)$.
Let $((s_k, x_k))_{k\geq 1}$ be a (measurable) enumeration of the atoms of $N_\theta$. Let $U_{\alpha,1}$, $U_{\alpha,2} \ldots $ denote independent copies of $U_\alpha$, which are independent of $N_\theta$. Put $N:= \sum_{k\geq 1}\delta_{(s_k, x_k+ U_{\alpha, k}(\cdot \wedge \sigma_k))}$, where
\[
\sigma_k:=\sigma(x_k+U_{\alpha, k}(\cdot)):=\inf\{t\geq 0\ :\ x_k+U_{\alpha, k}(t)=0\}.
\]
Then $N$ is a Poisson random measure on $[0,\infty)\times D$ with intensity ${\rm d}s\times \hat P= {\rm d}s\times  \bar P^\theta$.


Without loss of generality we can and do assume that $X(0)=0$ and that the process $X$ is built upon the Poisson point measure $N$ with the help of It\^{o}'s procedure. Here are details of the construction. Put
\[
\tau(s):=\sum_{s_k\leq s} \sigma(x_k+U_{\alpha, k}(\cdot)),\quad s\geq 0,
\]
and
\[
\vf(t):= \inf\{s\geq 0\  : \ \tau(s)> t\},\quad t\geq 0.
\]
For $t\in [\tau(s_k-), \tau(s_k))$, put $X(t):= x_k+ U_{\alpha, k}(t- \tau(s_k-))$, and, for $t\notin \bigcup_k[\tau(s_k-), \tau(s_k))$, put $X(t):=0$.
\begin{lem}\label{lem: 633}
For almost all $t>0$ with respect to the Lebesgue measure,
\[
\Pb\Big(\exists k \ :\ \ t\in (\tau(s_k-), \tau(s_k)) \ \ \mbox{and}\ \
\tau(s_k-) =\vf(t)\Big)=1.
\]
\end{lem}
The proof follows from the fact that $X$ has a zero sojourn at $0$ and Fubini's theorem.

Put
\bel{eq: trunc_measure_Poisson}
N_\theta^{(\ve)}:= \sum_k \1_{\{|x_k|>\ve\}} \delta_{(s_k, x_k)},\ \ \
N^{(\ve)}:= \sum_{k\geq 1}\1_{\{|x_k|>\ve\}} \delta_{(s_k, x_k+ U_{\alpha, k}(\cdot \wedge \sigma_k))}.
\ee
Then $N_\theta^{(\ve)}$ is a Poisson point measure on $[0,\infty)\times (\mbR\setminus\{0\})$ with intensity ${\rm d}s \times (\1_{\{|x|>\ve\}}\theta({\rm d}x))$, and $N^{(\ve)}$ is a Poisson point measure on $[0,\infty)\times D$ with intensity ${\rm d}s \times \bar P^{\1_{\{|x|>\ve\}}\theta({\rm d}x)}$.

Assuming that $m=0$ we now construct a process $X^{(\ve)}$ with the help of the Poisson random measure $N^{(\ve)}$, along the lines of construction of $X$ with the help of $N$. Put
\[
\tau^{(\ve)}(s):=\sum_{s_k\leq s,\ |x_k|>\ve} \sigma(x_k+U_{\alpha, k}(\cdot)),\quad s\geq 0
\]
and 
\[
\vf^{(\ve)}(t):= \inf\{s\geq 0\  : \ \tau^{(\ve)}(s)\geq t\},\quad t\geq 0.
\]
For $t\in [\tau^{(\ve)}(s_k-), \tau^{(\ve)}(s_k))$, put
$X^{(\ve)}(t):= x_k+ U_{\alpha, k}(t- \tau^{(\ve)}(s_k-))$ provided that $|x_k|>\ve$. Otherwise, we put $X^{(\ve)}(t):=0$.
\begin{remk}
In contrast to the process $X$, a.s.\ there does not exist a $t$ such that $X^{(\ve)}(t)=0$. The process $X^{(\ve)}$ jumps upon `touching' $0$.
\end{remk}

In view of $\int_\mbR \int_0^T   \1_{\{|x|>\ve\}} {\rm d}s  \eta({\rm d}x)<\infty$ for any $T>0$, the Poisson point process  $N_\eta^{(\ve)}$ can be represented by $\sum_{k\geq 1}\delta_{(s^{(\ve)}_k, x^{(\ve)}_k)}$, where  $0<s^{(\ve)}_1<s^{(\ve)}_2<s^{(\ve)}_3<\ldots$, and each interval $[0,T]$ contains finitely many $s_k^{(\ve)}$ a.s. Put  $\sigma_k^{(\ve)}:=\sigma(x^{(\ve)}_k+U^{(\ve)}_{\alpha, k}(\cdot))$, where $U^{(\ve)}_{\alpha, k}$ is a process from the collection $(U_{\alpha, j})$, which corresponds to $s_k^{(\ve)}$ in the representation of $N$. According to the construction of $X^{(\ve)}$,
\[
X^{(\ve)}(0)=x_1^{(\ve)},
\]
\[
 X^{(\ve)}(\sigma_{1}^{(\ve)}+\ldots+\sigma_{k}^{(\ve)})=x^{(\ve)}_{k+1},
\]
\[
 X^{(\ve)}(t)=x^{(\ve)}_{k}+ U^{(\ve)}_{\alpha,k}(t),
 t\in[\sigma_{1}^{(\ve)}+\ldots+\sigma_{k}^{(\ve)}, \sigma_{1}^{(\ve)}+\ldots+\sigma_{k}^{(\ve)}+\sigma_{k+1}^{(\ve)}),
\]
and
\[
\sigma_{k+1}^{(\ve)}=\inf\{ t\geq \sigma_{1}^{(\ve)}+\ldots+\sigma_{k}^{(\ve)}\ \ : \ \ x_k^{(\ve)} +
U_{\alpha, k}^{(\ve)}(t- (\sigma_{1}^{(\ve)}+\ldots+\sigma_{k}^{(\ve)})) =0\}.
\]
Plainly, the process $U^{(\ve)}_\alpha$ defined by
$
U^{(\ve)}_\alpha(0):=0,
$ and
\[
 U^{(\ve)}_\alpha(t):=
U^{(\ve)}_\alpha(\sigma_{1}^{(\ve)}+\ldots+\sigma_{k}^{(\ve)})+ U_{\alpha, k}^{(\ve)}(t- (\sigma_{1}^{(\ve)}+\ldots+\sigma_{k}^{(\ve)})),
\]
for $t\in[\sigma_{1}^{(\ve)}+\ldots+\sigma_{k}^{(\ve)}, \sigma_{1}^{(\ve)}+\ldots+\sigma_{k}^{(\ve)}+\sigma_{k+1}^{(\ve)})
$, is a symmetric $\alpha$-stable L\'{e}vy process, which is independent of $N_\theta^{(\ve)}$. In particular, $U^{(\ve)}_\alpha$ is independent of
\bel{eq:beta_trunc_stb}
S_\theta^{(\ve)}(s):=\int_{[0,\,s]} \int_{|x|>\ve} x N_\theta^{(\ve)} ({\rm d}z, {\rm d}x) = \int_{[0,\,s]} \int_{\mbR\setminus\{0\}} x
\1_{\{|x|>\ve\}} N_\theta({\rm d}z, {\rm d}x).
\ee
Note also that
\[
 X^{(\ve)}(t) = U^{(\ve)}_\alpha(t)+ S_\theta^{(\ve)}(\vf^{(\ve)}(t)),\quad t\geq 0.
\]

%
\begin{remk}
Let $\zeta_\ve$ be a random variable with distribution $\Pb(\zeta_\ve\in A)=\frac{\int_{|x|>\ve}\1_A(x) \theta({\rm d}x)}{\theta({|x|>\ve})}$. Then the distribution of $X_\ve$ coincides with that of $X_{\zeta_\ve}$, see \eqref{eq:def_of_X} for the definition.
\end{remk}
\begin{lem} \label{lem: 634}
For each $T>0$, almost surely
\[
\lim_{\ve\to 0}\sup_{s\in [0,\,T]} |\tau^{(\ve)}(s)-\tau(s)| =0,
\]
\[
\lim_{\ve\to0}\sup_{t\in [0,\,T]} |\vf^{(\ve)}(t)-\vf(t)| =0
\]
and
\[
\lim_{\ve\to0}\sup_{s\in [0,\,T]} |S^{(\ve)}_\theta(s)-S_\theta(s)| =0,
\]
where
 $S_\theta(s)= \int_0^s \int_{\mbR\setminus\{0\}}xN_\theta({\rm d}z, {\rm d}x)$ for $s\geq 0$.
\end{lem}

The proof follows from the construction and the fact that the processes $(\tau(t))$ and $S_\theta$ are pure-jump L\'{e}vy processes of (locally) finite variation.
\begin{lem}\label{lem: 636}
For almost all $t>0$ with respect to the Lebesgue measure,
\bel{eq:1207}
\Pb(\exists \ve_0>0\ \forall \ve\in (0,\ve_0):\ \ \vf^{(\ve)}(t)=\vf(t))=1.
\ee
\end{lem}
\begin{proof}
It follows from Lemmas \ref{lem: 633} and \ref{lem: 634} that
for a.e.\ $t>0$ with probability $1$ there exists $k$ such that
\[
\ t\in (\tau(s_k-), \tau(s_k)) \ \ \mbox{and}\ \
\tau(s_k-) =\vf(t)
\]
and
\[
\lim_{\ve\to0}\tau^{(\ve)}(s_k-)= \tau(s_k-),\ \ \lim_{\ve\to0}\tau^{(\ve)}(s_k)= \tau(s_k).
\]
These entail $t\in (\tau^{(\ve)}(s_k-), \tau^{(\ve)}(s_k))$ for small $\ve>0$ and thereupon \eqref{eq:1207}.
\end{proof}
\begin{corl}\label{cor:722}
For each $t>0$,
\[
\lim_{\ve\to0}S_\theta^{(\ve)}(\vf^{(\ve)}(t)) = S_\theta(\vf(t))\quad \text{a.s.}
\]
\end{corl}
A fixed time $t_0$ is not a jump-time of a L\'{e}vy process a.s. Hence, by Fubini's theorem, for each $k\geq 1$ and almost all $t_0>0$ with respect to the Lebesgue measure,
\[
\Pb\left(\lim_{t\to t_0} U_{\alpha, k}(t)=U_{\alpha, k}(t_0)\right)=1\quad\text{a.s.}
\]
This together with Lemmas \ref{lem: 633} and \ref{lem: 634} enables us to conclude that
\[
\Pb\left(\lim_{t\to t_0} X^{(\ve)}(t)=X(t_0)\right)=1
\]
for almost all $t_0>0$. Invoking Corollary \ref{cor:722} we infer
\bel{eq:741}
\lim_{t\to t_0} U^{(\ve)}_\alpha(t)
 = \lim_{t\to t_0} \left( X^{(\ve)}(t) - S_\theta^{(\ve)}(\vf^{(\ve)}(t))\right)=
 X(t_0) - S_\theta(\vf(t_0))\quad\text{a.s.}
\ee
for almost all $t_0>0$.


Denote the process $(X(t) - S_\theta(\vf(t)))$ by $U_\alpha$. The paths of this process are c\`{a}dl\`{a}g, for so are the paths of $X$
and $(S_\theta(\vf(t)))$. The process $U_\alpha$ is a symmetric $\alpha$-stable L\'{e}vy process as a limit of  symmetric $\alpha$-stable L\'{e}vy processes.
Hence, \eqref{eq:741} entails
\[
X(t) = U_\alpha(t)+ S_\theta(\vf(t)), \quad t\geq 0\quad\text{a.s.}
\]
 Since $(U^{(\ve)}_\alpha(t))$ is independent of $S_\theta$, so is $U_\alpha$.



\end{proof}

\begin{remk}
We have not proved that $\lim_{\ve \to 0} X^{(\ve)}=X$ a.s.\ on $D$. However, this has never been claimed. 
\end{remk}
\begin{remk}
For each $T>0$, the process $(U^{(\ve)}_\alpha(t))_{t\in[0,T]}$ was obtained by merging finitely (but randomly) many fragments of paths of independent $\alpha$-stable L\'{e}vy processes. However, the process $(U_\alpha(t))_{t\in[0,T]}$ is built upon a countable number of paths. We cannot offer a good formula for $U_\alpha(t)$, other than $U_\alpha(t)=X(t)-S_\theta(\vf(t))$ and attract the reader attention to the fact that the set $\{t\geq 0 : X(t)=0\}$ is uncountable (the set does not coincide with a countable set of the endpoints of excursion intervals). A similar noise  representation, in terms of the difference of a Markov process and a generalized drift, is typical in the framework of diffusions with semipermeable membrane, see \cite{Portenko90,portenko1976generalized}. In this context the Markov process $Y$, say is first constructed via a semigroup technique. Then the generalized drift is identified with an additive functional of $Y$. Finally, it has to be checked that the difference of $Y$ and the generalized drift is a stochastic integral.
\end{remk}


The L\'{e}vy measure $\theta$ of the process $S_\theta$ appearing in \eqref{eq:SDE_Levy} satisfies $\E^\theta (1-e^{-\sigma})=1$. Below we construct a process, similar to $S_\theta$, which satisfies \eqref{eq:SDE_Levy} when the latter equality does not necessarily holds.
 Denote by $(\rho_t)$, $\hat P_{\alpha}$ and $ N_{\alpha}$ the entrance law, the intensity of excursion measure and the corresponding Poisson point measure of a symmetric $\alpha$-stable L\'{e}vy process.
 Recall (see Lemma \ref{lem:360} and Remark \ref{remk:360}) that there exists a unique continuous entrance law with a zero sojourn at $0$, and that this entrance law corresponds to a symmetric $\alpha$-stable
 L\'{e}vy process. Let $\theta$ be a sigma-finite measure on $\mbR\setminus\{0\}$ satisfying $\E^\theta (1-e^{-\sigma})=1$.
Define the entrance law $\theta_t:=\theta P^0_t$, the characteristic measure $\bar P^\theta$ and the corresponding Poisson point measure $N$, which is independent of $N_{\alpha}$. For $p\in[0,1]$, define the entrance law $p\theta_t+(1-p)\rho_t$, the intensity of excursions measure $p \bar P^\theta+ (1-p)\hat P_{\alpha}$ and the Poisson point measure $N(p\, {\rm d}t, {\rm d}u)+ N_{\alpha}((1-p)\, {\rm d}t, {\rm d}u)$, see \eqref{eq:entrance-general_decomposition} and \eqref{eq:characteristic-general_decomposition}. Using these objects we now construct a Feller process $X$ with the help of  It\^{o}'s procedure. Since $(p \bar P^\theta+ (1-p) \hat P_{\alpha})(1-e^{-\sigma})=1$, the process $X$ has a zero sojourn at $0$.



%
%
\begin{thm}\label{thm:General_SDE_SKEW}
For each $p\in [0,1]$, the process $X$ is a (weak) solution to the SDE
\bel{eq:SDE_Levy_1}
 X(t) =  U_\alpha(t) +  S_\theta( p L^X_0(t)),\quad t\geq 0,
\ee
where $U_\alpha$ is a symmetric $\alpha$-stable L\'{e}vy process, $S_\theta$ is a pure-jump L\'{e}vy process with the L\'{e}vy measure $\theta$, which is independent of $U_\alpha$, and the process $L^X_0$ is the local time of $X$ at $0$.
\end{thm}
\begin{remk}
The process $U_\alpha$ appearing in \eqref{eq:SDE_Levy_1} is different from the $\alpha$-stable L\'{e}vy process which is built upon $N_{\alpha}$ alone.
\end{remk}
\begin{remk}
The process $S_\theta( p\cdot)$ has the same distribution as $S_{p\theta}(\cdot)$. As a consequence, $X$ is also a solution to
\[
X(t) =  U_\alpha(t) +  S_{p\theta}(L^X_0(t)),\quad t\geq 0.
\]
\end{remk}

Theorem C follows from the last remark specialized to $\theta=\eta^\ast$ (so that $S_\theta=S_\beta$), where $\theta^\ast$ is the measure defined in \eqref{eq:Levy measure}.
\begin{proof}[Proof of Theorem \ref{thm:General_SDE_SKEW}]
Define  the Poisson random measures $N^{(\ve)}_\eta$ and $ N^{(\ve)}$ via
\eqref{eq: trunc_measure_Poisson}.

 Put $\theta^{(\ve)}({\rm d}x):=\1_{\{|x|>\ve\}}\theta({\rm d}x)$, $\theta^{(\ve)}_t:=\theta^{(\ve)}P^0_t$ and
\[
q_\ve:=\frac{1-p  \int_{|x|>\ve} \E^x(1-e^{-\sigma}) \theta({\rm d}x)}{\hat P_\alpha(1-e^{-\sigma})}.
\]
Use now It\^{o}'s procedure to build the process $X^{(\ve)}$ upon the Poisson point measure $N^{(\ve)}(p\, {\rm d}t, {\rm d}u)+ N_{\alpha}(q_\ve\, {\rm d}t, {\rm d}u)$, which corresponds to the entrance law $p \theta^{(\ve)}_t+  q_\ve \rho_t$ and the characteristic measure $p\bar P^{\1_{\{|x|>\ve\}}\theta({\rm d}x)}+ q_\ve \hat P_\alpha$. It follows from the construction that $X^{(\ve)}$ has a zero sojourn at $0$, and that
\[
X^{(\ve)}(t) = U^{(\ve)}_\alpha(t)+ S_\theta^{(\ve)}(p\vf^{(\ve)}(t)),\quad t\geq 0,
\]
where the process $S_\theta^{(\ve)}$ is defined in \eqref{eq:beta_trunc_stb}. The remainder of the proof mimics that of Theorem \ref{thm:eq_Levy_1}.

We proceed with a comment.

%
%
%
%
Assume that we have constructed the process $\tilde X^{(\ve)}$ with the help of the Poisson point measure $N^{(\ve)}(p_\ve\, {\rm d}t, {\rm d}u)+ N_{\alpha}(q {\rm d}t, {\rm d}u)$, where
\[
p_\ve=\frac{p \int_{\mbR} \E^x(1-e^{-\sigma}) \theta({\rm d}x)  }{\int_{|x|>\ve} \E^x(1-e^{-\sigma}) \theta({\rm d}x)}.
\]
Then
\[
\tilde X^{(\ve)}(t) = \tilde  U^{(\ve)}_\alpha(t)+ S_\theta^{(\ve)}(p_\ve \tilde \vf^{(\ve)}(t)),\quad t\geq 0.
\]
It is likely that the limit relation $\lim_{\ve\to 0} S_\theta^{(\ve)}(p_\ve \tilde \vf^{(\ve)}(t))= S_\theta (p \vf (t))$ a.s.\ also holds. However, this fact does not follow directly from
Lemmas \ref{lem: 634} and \ref{lem: 636}.
 Indeed, the composition functional is not continuous on
 $D\times C_\uparrow([0,\infty))$, where $C_\uparrow([0,\infty))$ is the set of nondecreasing continuous functions, see \cite[Section 13.2]{Whitt}. The composition $(S_\theta,  \vf)\mapsto S_\theta \circ \vf$ would be a.s.\ continuous, if the function $\vf$ were a.s.\ strictly increasing, which is not the case. The local time $\vf$ has many intervals of constancy, which are directly connected to the points of growth of $S_\theta$.
\end{proof}

\subsection{Uniqueness for SDEs with local times for the skew Levy  process}\label{sec:Uniqueness}



\begin{thm}\label{thm:General_SDE_SKEW_uniqueness}
%
Assume that there exists a filtered probability space $(\Omega, (\cF_t), \cF, \Pb),$ a $(\cF_t)$-adapted symmetric $\alpha$-stable process $U_\alpha$ with $\alpha\in (1,2)$ and a homogeneous Feller process $(X, (\cF_t))$ with a zero sojourn at $0$ which satisfy
\[
 X(t) =  U_\alpha(t) +  H(L^X_0(t)),\quad t\geq 0,
\]
where $H$ is a L\'{e}vy process of bounded variation, which is independent of $U_\alpha$. Then there exist $p\in[0,1]$ and a sigma-finite measure $\theta$ on $\mbR\setminus\{0\}$ satisfying $\E^\theta (1-e^{-\sigma})=1$, for which the processes $(X,H)$ has the same distribution as the processes $(X(\cdot),S_\theta(p\cdot))$ appearing 
in Theorem \ref{thm:General_SDE_SKEW}.
\end{thm}

Theorem D is a specialization of Theorem \ref{thm:General_SDE_SKEW_uniqueness}, with the L\'{e}vy measure $\theta=\eta$ given by \eqref{eq:Levy measure} in which case $S_\theta=S_\beta$.
\begin{remk}
Under the assumptions of Theorem \ref{thm:General_SDE_SKEW_uniqueness} imposed on $(X, U_\alpha)$ there exists neither a solution to the equation
\[
{\rm d}X(t) = {\rm d} U_\alpha(t) + \gamma {\rm d}L^X_0(t),\quad t\geq 0
\]
for $\gamma\neq0$, nor a solution to \eqref{eq:SDE_Levy_1} for $p>1$.
\end{remk}
\begin{proof}
Since the local time only grows on the set $\{t\geq 0\ : \ X(t)=0\}$, the process $X$ is a recurrent extension of a symmetric $\alpha$-stable L\'{e}vy process killed at $0$. Hence, there exists a unique value of $p\in[0,1]$ and a unique measure $\theta$ satisfying $\E^\theta (1-e^{-\sigma})=1$, for which the entrance law of $X$ is given by $(1-p) \rho_t+ p \, \theta P^0_t$, see \eqref{eq:entrance-general_decomposition} and \eqref{eq:characteristic-general_decomposition}.
It follows from Theorem \ref{thm:General_SDE_SKEW} that there exists a version $\wt X$ of $X$ satisfying
\[
\wt X(t) =  \wt U_\alpha(t) +  \wt S_\theta (p L^{\wt X}_0(t)),\quad t\geq 0.
\]
Observe that the processes $X$ and $\wt X$ are semimartingales. Since the local times $(L^X_0(t))$ and $(L^{\wt X}_0(t))$ are a.s.\ continuous processes which do not increase a.s.\ on the sets $\{s\geq 0\ :\ X(s)\neq 0\}$ and $\{s\geq 0\ :\ \wt X(s)\neq 0\}$, respectively, we conclude that, for any $\delta>0$, with probability $1$,
\[
\int_{[0,\,t]}\1_{\{|\wt X(s-)|\geq \delta\}} {\rm d}  (\wt S_\theta( p L^{\wt X}_0(s)))
= 0,\quad t>0
\]
and
\[
\int_{[0,\,t]}\1_{\{|X(s-)|\geq \delta\}}{\rm d} (  H(L^X_0(s)))=0,\quad t\geq 0.
\]
This entails
$$\int_{[0,\,t]}\1_{\{|X(s-)|\geq \delta\}} {\rm d} X(s)=\int_{[0,\,t]} \1_{\{|X(s-)|\geq \delta\}} {\rm d}  U_\alpha(s)
$$
and
\[
\int_{[0,\,t]} \1_{\{|\wt X(s-)|\geq \delta\}} {\rm d}\wt X(s)=\int_{[0,\,t]}\1_{\{|X(s-)|\geq \delta\}} {\rm d}\wt U_\alpha(s).
\]
Sending $\delta\to0+$ we conclude that there exists a measurable function $F:D\to D$ satisfying
$\wt U_\alpha=F(\wt X)$ and $U_\alpha=F(X)$ a.s. Since the process $X$ has the same distribution as $\wt X$, there exists a measurable function $G:D\to D$ satisfying $\wt L^{\wt X}_0=G(\wt X)$ and $L^{X}_0=G(X)$, see \cite[Section III.3]{Blumenthal2012excursions}. As a consequence, the distributions of the pairs $(H(L^{\wt X}_0), L^{\wt X}_0)$ and $(  S_\theta( p L^{ X}_0), L^{ X}_0)$ are the same. This, in its turn, ensures that the distributions of the L\'{e}vy processes $(H(t))$ and $(S_\theta( p t))$ are the same.
%
%
\end{proof}
{\bf Acknowledgement.}
The present work was supported by the National Research Foundation of Ukraine (project 2020.02/0014 `Asymptotic regimes of perturbed random walks: on the edge of modern and classical probability'). The authors are very grateful to M.~I. Portenko for numerous helpful discussions. 

\bibliographystyle{plain}

\end{document}